\documentclass{article}

\usepackage{amsmath}
\usepackage{amssymb}
\usepackage{amsthm, amscd}
\usepackage{amsmath,amsthm,amsfonts,amssymb}

\usepackage[usenames]{color}

\newtheorem{thm}{Theorem}[section]
\newtheorem{prop}[thm]{Proposition}
\newtheorem{lem}[thm]{Lemma}
\newtheorem{cor}[thm]{Corollary}
\newtheorem{defn}[thm]{Definition}

\theoremstyle{remark}
\newtheorem{rem}[thm]{\bf Remark}

\newcommand{\la}{\mathcal}
\newcommand{\mbb}{\mathbb}
\newcommand{\fra}{\mathfrak}
\newcommand{\p}{\fra{p}}
\newcommand{\q}{\fra{q}}

\newcommand{\ds}{\displaystyle}
\newcommand{\ov}{\overline}

\newcommand{\N}{\mathbb{N}}
\newcommand{\Z}{\mbb{Z}}

\newcommand{\Q}{\mbb{Q}}
\newcommand{\D}{\mathcal{D}}

\newcommand{\define}{\stackrel{\text{def}}{=}}

\numberwithin{equation}{section}
\usepackage{fancyhdr}

\begin{document}
\title{On elementary equivalence of rings with a finitely generated additive group}

\author{ Alexei Miasnikov,
 Mahmood Sohrabi\footnote{Address: Stevens Institute of Technology, Department of Mathematical Sciences, Hoboken, NJ 07087, USA. Email: msohrab1@stevens.edu }}

\maketitle
\begin{abstract}In this paper we provide a complete algebraic characterization of elementary equivalence of rings with a finitely generated additive group in the language of pure rings. The rings considered are arbitrary otherwise.
  \\
{\bf 2010 MSC:} 03C60\\
{\bf Keywords:} Ring, Elementary Equivalence,  Largest Ring of a Bilinear Map
\end{abstract}

\section{Introduction}
This paper continues the authors' efforts~\cite{MS, MS2010, MS2016}, in providing a comprehensive and uniform approach to various model-theoretical questions on algebras and nilpotent groups. By a \emph{scalar ring} we mean a commutative associative unitary ring. Assume $A$ is a scalar ring. We say that $R$ is an \emph{$A$-algebra} if $R$ is abelian group equipped with an $A$-bilinear binary operation. We use the term \emph{ring} for a $\Z$-algebra where $\Z$ is the ring of rational integers, reserving the term scalar ring for commutative associative unitary rings. The ring $R$ is said to be a \emph{finite dimensional $\Z$-algebra} or an \emph{FDZ-algebra} for short if the additive group $R^+$ of the ring $R$ is finitely generated as an abelian group.  
 
The main problem we tackle here is to characterize the elementary equivalence of  FDZ-algebras via a complete set of elementary invariants. The invariants will be purely algebraic.

\subsection{Statements of the main results}\label{approach:sec}
In this paper the language $L$ denotes the language of pure rings without a constant for multiplicative identity. That is because an arbitrary ring may not have a unit. By $L_1$ we mean the usual language of rings with identity. 

 For us an $A$-module $M$ is a two-sorted structure $\langle M, A, s \rangle$, where $M$ is  an abelian group, $A$ is a scalar ring and $s$ is the predicate describing the action of $A$ on $M$. Denote the language by $L_2$. We often drop $s$ from our notation. Since the scalar ring are always assumed to be commutative we do not specify whether the modules are left or right modules.  

Here is our first main result.

 \begin{thm}\label{elemmod:thm}Let $A$ be an FDZ-scalar ring and let $M$ be a finitely generated $A$-module. Then
 there exists a sentence $\psi_{M,A}$ of the language $L_2$ such that $\langle M,A\rangle\models \psi_{M,A}$ and for
 any FDZ-scalar ring $B$ any finitely generated $B$-module $N$, we
 have
 $$\langle N,B\rangle\models \psi_{M,A} \Leftrightarrow \langle N,B\rangle
 \cong \langle M,A\rangle.$$\end{thm} 
 
 The proof of the theorem appears at the end of Section~\ref{Z-interpret:sec}. Indeed Theorem~\ref{elemmod:thm} implies the next three statements. The first two state the same result.  
 \begin{cor}\label{scalarrings:cor} For any FDZ-scalar ring $A$ there exists a formula $\psi_A$ of $L_1$
such that $A\models \psi_A$ and for FDZ-scalar ring $B$ we have
$$B\models \psi_A \Leftrightarrow A\cong B.$$\end{cor}
\begin{cor}\label{scalarrings:cor2}Let $\mathcal{K}$ be the class of all FDZ-scalar rings. Then any $A$
from $\mathcal{K}$ is finitely axiomatizable inside
$\mathcal{K}$.\end{cor}
Let us denote by $L_3$ the first-order language of two-sorted algebras. An algebra $\langle  C, A\rangle$ consists of an arbitrary ring $C$, and the scalar ring $A$ (and a predicate describing the scalar multiplication which is dropped from the notation). As mentioned it is actually a corollary of Theorem~\ref{elemmod:thm}. We provide a brief of it at the beginning of Section~\ref{main:sec}.  
\begin{thm}\label{elem-iso-alg:thm} Let $\mathcal{A}$ be the class of all two-sorted algebras $\langle C, A\rangle$ where $C$ is finitely generated as an $A$-module and $A^+$ is finitely generated as an abelian group. For each  $\langle C, A\rangle\in \mathcal{A}$ there exists a formula $\phi_{C,A}$ of $L_3$ such that $\langle C, A\rangle\models \phi_{C,A}$ and for any $\langle D, B\rangle \in \mathcal{A}$,
 $$\langle D, B\rangle\models \phi_{C,A} \Leftrightarrow \langle C,A\rangle \cong \langle D,B\rangle$$
 as two-sorted algebras.    \end{thm}

 To state the main result of the paper we need to introduce some more definitions and notation. Consider an arbitrary ring $R$. Define the \emph{two-sided annihilator ideal} of $R$ by $$Ann(R)=\{x\in R: xy=yx=0, \forall y\in R\}.$$ By $R^2$ we denote the ideal of $R$ generated by all products $x\cdot y$ (or $xy$ for short) of elements of $R$. 
 
Consider a scalar ring $A$ and let $R$ be an $A$-algebra. Assume $I$ is an ideal of $R$. Let
 $$Is_A(I)\define \{ x\in R: ax\in I,\text{ for some }a\in A\setminus\{0\}\}.$$
 It is easy to show that $Is_A(I)$ is an ideal in $R$. We simply denote $Is_\Z(I)$ by $Is(I)$. 
 
 Now assume $R$ is an FDZ-algebra. An \emph{addition} $R_0$ of $R$ is a direct complement of the ideal $\Delta(R)\define Is(R^2)\cap Ann(R)$ in $Ann(R)$. Such a complement exists in this situation since $Ann(R)$ is a finitely generated abelian group and $Ann(R)/\Delta(R)$ is free abelian. It is clear that $R_0$ is actually an ideal of $R$. The quotient $R_F\define R/R_0$ is called the \emph{foundation} of $R$ associated to the addition $R_0$.
 
 Finally for an FDZ-algebra $R$ set $M(R)\define Is(R^2+Ann(R))$ and $N(R)\define Is(R^2)+Ann(R)$. Note that $M(R)/N(R)$ is a finite abelian group.
  
 We are now ready to state the main result of this paper.
 
 \begin{thm}\label{mainnice:thm} Assume $R$ and $S$ are FDZ-algebras. Then the following are equivalent.

 \begin{enumerate}
 \item $R\equiv S$ as arbitrary rings.
 \item Either $M(R)=N(R)$ and $R\cong S$, or there exists a monomorphism $\phi:R\to S$ of rings and additions $R_0$ and $S_0$ of $R$ and $S$ respectively such that 
 \begin{enumerate}
 
 \item $\phi$ induces an isomorphism $R/R_0\cong S/S_0$,
 \item $\phi$ induces an isomorphism $\dfrac{M(R)}{N(R)}\cong \dfrac{M(S)}{N(S)}$,
 \item $\phi$ restricts to a monomorphism from $R_0$ into $S_0$, where the index $[S_0:\phi(R_0)]$ is finite and prime to the index $[M(R):N(R)]\neq 1$. 
 \end{enumerate}\end{enumerate}\end{thm}
  The direction $(1.)\Rightarrow (2.)$ is called the \emph{Characterization Theorem} and will be proved in Section~\ref{main:sec}. The direction $(2.)\Rightarrow (1.)$ called naturally the \emph{converse of the characterization theorem}, stated in somewhat different terms, is given by Theorem~\ref{converse}. 
  
  An FDZ-algebra is called \emph{regular} if for some addition (and therefore for any addition) $R_0$ there exists a subring $R_F$ of $R$ containing $R^2$ such that $R\cong R_F \times R_0$. In Lemma~\ref{regular-M=N:lem} we shall prove that $R$ is a regular FDZ-algebra if and only if $M(R)=N(R)$. So the following statement was actually embedded in Theorem~\ref{mainnice:thm}.
   
\begin{cor}\label{regular:cor} Let $R$ be a regular FDZ-algebra. Then for an FDZ-algebra $S$,
$$R\equiv S \Leftrightarrow R\cong S.$$
\end{cor}   

Finally we call an FDZ-algebra $R$ \emph{tame} if $Ann(R)\leq Is(R^2)$. The following theorem is the generalization of Corollary~\ref{scalarrings:cor2} to the class $\la{T}$ of all tame FDZ-algebras.

\begin{thm}\label{tame:thm} Let $\la{T}$ be the class of all tame FDZ-algebras. Then any $R$ from $\la{T}$ is finitely axiomatizable inside the class of all FDZ-algebras.\end{thm} 

\subsection{Our approach} 
 Let us give an informal account of our methods in proving Theorem~\ref{mainnice:thm}. Recall that an arbitrary ring $R$ is an abelian group together with a bilinear map:
 $$f_R:R\times R \to R, \quad (x,y)\mapsto xy.$$
 The bilinear map $f$ induces a full non-degenerate map 
  $$f_{RF}: \frac{R}{Ann(R)}\times \frac{R}{Ann(R)} \to R^2.$$
 By Theorem~\ref{ringinter}, from Section~\ref{bilin} there exists a canonical scalar ring $P(f_R)$ of $f_{RF}$ and its actions on $R/Ann(R)$ and $R^2$ are interpretable in $f_R$. Moreover the largest subring $A(R)$ of $P(f_R)$ consisting of those $\alpha$ making the canonical homomorphism:
  $$\eta: R^2 \to \frac{R}{Ann(R)}$$ 
 $A(R)$-linear is a definable subring of $R$. So indeed the two-sorted algebras $\langle R/Ann(R), A(R)\rangle$ and $\langle R^2, A(R)\rangle$ are both interpretable in the pure ring $R$. Then the main theorem will follow from Theorem~\ref{elem-iso-alg:thm} and a few other technical results.  Bilinear maps and the relevant terminology will be discussed in Section~\ref{bilin}. Theorem~\ref{mainnice:thm} will be proved in Section~\ref{main:sec}. The converse Theorem~\ref{converse} of Theorem~\ref{mainnice:thm} will appear in Section~\ref{converse:sec}, thereby providing a complete algebraic characterization of elementary equivalence of FDZ-algebras.  

\subsection{Organization of the paper}  
We finish the introduction by describing the organization of the paper. In Section~\ref{pre:sec} we provide background and describe our notation. In particular we shall review logical notation and background, bilinear maps and their model theory and finally a little bit of algebras. In Section~\ref{Z-interpret:sec} we shall discuss FDZ-scalar rings (associative commutative and unitary), resulting in proofs of Theorem~\ref{elemmod:thm} and Corollary~\ref{scalarrings:cor}. In Section~\ref{main:sec} we obtain a necessary condition for elementary equivalence of arbitrary FDZ-algebras and provide a proof of the characterization direction of Theorem~\ref{mainnice:thm} as well as a proof of Theorem~\ref{tame:thm}. In Section~\ref{converse:sec} we prove Theorem~\ref{converse} which indeed proves the converse of the characterization theorem.

\section{Preliminaries}\label{pre:sec}
 
 \subsection{Bilinear maps} 
 Assume $M_1$, $M_2$ and $N$ are $A$-modules, where $A$ is a commutative associative ring with unit. The map 
 $$f:M_1\times
M_2\rightarrow N$$
is called $A$-bilinear if 
$$f(ax,y)=f(x,ay)=af(x,y)$$ 
for all $x\in M_1$, $y\in M_2$ and $a \in A$.
 An $A$-bilinear mapping $f:M_1\times M_2\rightarrow N$ is called \textit{non-degenerate in the first variable} if $f(x,y)=0$ for all $y$ in $M_2$ implies $x=0$. Non-degeneracy with respect to the second variable is defined similarly. The mapping $f$ is called \textit{non-degenerate} if it is non-degenerate with respect to first and second variables.
We call the bilinear map $f$, a \textit{full bilinear mapping} if $N$ is generated as an $A$-module by elements $f(x,y)$, $x\in M_1$ and $y\in M_2$.

\subsection{Preliminaries on logic}
For the most part we follow standard model theory texts such as~\cite{hodges} regarding notation and model theory. An arbitrary ring $R$ is a structure with signature $\langle +, \cdot, 0 \rangle$ and with the corresponding language is called $L$. A scalar ring $A$ is a structure with signature $\langle +, \cdot, 0,1 \rangle$ and the corresponding language is called $L_1$. 

\subsubsection{\protect Interpretations}\label{ss1}
\label{interpret1}Let $\mathfrak{B}$ and $\mathfrak{U}$ be  structures of signatures $\Delta$ and
$\Sigma$ respectively. We may assume that $\Sigma$ and $\Delta$ do not contain any function symbols replacing them if necessary with predicates (i.e. replacing operations with their graphs). The structure $\mathfrak{U}$ is said to be
\textit{interpretable}\index{interpretable} in $\mathfrak{B}$ with parameters $\bar{b}\in |\mathfrak{B}|^m$ or \textit{relatively
interpretable} in $\mathfrak{B}$ if there is a set of first-order formulas
$$\Psi=\{A(\bar{x},\bar{y}), E(\bar{x},\bar{y}_1,\bar{y}_2),\Psi_{\sigma}(\bar{x}, \bar{y}_1, \ldots ,
\bar{y}_{t_{\sigma}}): \sigma \textrm{  a predicate of signature  } \Sigma \}$$ of signature $\Delta$ such
that
 \begin{enumerate}
 \item $A(\bar{b})=\{\bar{a}\in|\mathfrak{B}|^n:\mathfrak{B}\models A(\bar{b},\bar{a})\}$ is not empty,
 \item $E(\bar{x},\ov{y}_1,\ov{y}_2)$ defines an equivalence relation $\epsilon_{\bar{b}}$ on $A(\bar{b})$,
 \item if the equivalence class of a tuple of elements $\bar{a}$ from $A(\bar{b})$ modulo the
 equivalence relation $\epsilon_{\bar{b}}$ is denoted by $[\bar{a}]$, for every $n$-ary predicate
 $\sigma$ of signature $\Sigma$, the predicate $P_{\sigma}$ is defined on
 $A(\bar{b})/\epsilon_{\bar{b}}$ by
 $$P_{\sigma}([\bar{b}],[\ov{a}_1], \ldots, [\ov{a}_n])\Leftrightarrow_{\text{def}}\mathfrak{B}\models \Psi_{\sigma}(\bar{b}, \ov{a}_1,
 \ldots, \ov{a}_n),$$
 \item There exists a map $f:A(\bar{b})\rightarrow |\fra{U}|$ such that the structures $\mathfrak{U}$ and
 $\Psi(\mathfrak{B},\bar{b})=\langle A(\bar{b})/\epsilon_{\bar{b}},P_{\sigma}:\sigma\in \Sigma \rangle$ are isomorphic via the map $\tilde{f}:A(\bar{b})/\epsilon_{\bar{b}}\rightarrow |\fra{U}|$ induced by $f$.
 \end{enumerate}
 Let $\Phi(x_1,\ldots, x_n)$ be a first-order formula of signature $\Delta$. If $\mathfrak{U}$ is interpretable in $\mathfrak{B}$ for any 
  parameters $\bar{b}$ such that $\mathfrak{B}\models \Phi(\bar{b})$ then  $\mathfrak{U}$
 is said to be \textit{regularly interpretable} in $\mathfrak{B}$ with the help of the formula $\Phi$. If the tuple $\bar{b}$ is empty, $\mathfrak{U}$ is said
 be \textit{absolutely interpretable} in $\mathfrak{B}$. 
 
Now let $T$ be a theory of signature $\Delta$. Suppose that $S: Mod (T) \rightarrow K$ is a functor defined on the class $Mod (T)$ of all models of the theory $T$ (a category with isomorphisms) into a certain category $K$ of structures of signature $\Sigma$. If there exists a system of first-order formulas $\Psi$ of signature $\Delta$, which absolutely interprets the system $S(\frak{B})$ in any model $\frak{B}$ of the theory $T$ we say that $S(\frak{B})$ is \textit{absolutely interpretable in $\frak{B}$ uniformly with respect to $T$}.
 
 For example, the annihilator $Ann(R)$ of a ring $R$ is interpretable (or in this case definable) in $R$ uniformly with respect to the theory of groups. On the other hand, the ideal $R^2$, generally speaking, is not interpretable in $G$ uniformly with respect to the theory of groups. However, it is so if $R^2$ is of \emph{finite width} i.e. there is an $s\in \N$ such that
 $$R^2=\left\{\sum_{i=1}^sx_iy_i:x_i,y_i\in R\right\}.$$ For example in an FDZ-algebra $R^2$ is absolutely definable in $R$ uniformly with respect to $Th(R)$. Note that the ideal $R^2$ will have width less than or equal to $s$ if $R$ satisfies the first-order sentence \begin{equation}\label{R^2Sen:eqn} \begin{split}\phi_w:\forall x &\left((\exists x_1, \ldots, x_{s+1},y_1,\ldots ,y_{s+1}~ x=\sum_{i=1}^{s+1}x_iy_i)\right. \\ 
 &\left.\to (\exists z_1, \ldots, z_{s},t_1,\ldots ,t_{s}~ x=\sum_{i=1}^{s}t_is_i)\right).\end{split}\end{equation}

 \subsubsection{$A$-Modules as two-sorted structures} Assume $A$ is a scalar ring and $M$ is an $A$-module. For us the $A$-module $M$ is a two-sorted structure $M_A=\langle M, A, s\rangle$ where $A$ is a ring, $M$ is an abelian group and $s=s(x,y,z)$, where $x$ and $z$ range over $M$ and $y$ ranges over $A$ is the predicate describing the action of $A$ on $M$, that is $\langle M,A,s\rangle \models s(m,a,n)$ if and only if $a\cdot m =n$. Sometimes we drop the predicate $s$ from our notation and write $M_A=\langle M, A\rangle$. When we say that \emph{the ring $A$ and its action on $M$ are interpretable in a structure $\mathfrak{U}$} we mean that the one-sorted structure naturally associated to $M_A$ is interpretable in $\mathfrak{U}$. 
 
Note that if a multi-sorted structure has signature without any function symbols then there is a natural way to associate a one-sorted structure to it. We always assume that our signatures do not contain any function symbols, since functions can be interpreted as relations. Therefore when we talk about interpretability of multi-sorted structures in each other or interpretability of a multi-sorted structure into a one-sorted one we mean the interpretability of the associated one-sorted structures.  

Recall that a \emph{homomorphism $\theta : \langle M,A, s\rangle \to \langle N,B, t\rangle$ of two-sorted modules} is a pair $(\theta_1, \theta_2)$ where $\theta_1: M\to N$ is a homomorphism of abelian groups and $\theta_2:A \to B$ is a homomorphism of rings satisfying 
$$ s(m_1,a,m_2)\Leftrightarrow t(\theta_1(m_1), \theta_2(a), \theta_1(m_2)), \quad \forall a\in A, \forall m_1,m_2\in M.$$  A homomorphism $\theta$ as above is said be \emph{an isomorphism of two-sorted modules} if $\theta_1$ and $\theta_2$ are isomorphisms of the corresponding structures.

 \subsection{Largest ring of a bilinear map} \label{bilin}
In this section all the modules are considered to be faithful and scalar rings are always commutative associative with a unit. An $A$-module $M$ is said to be \emph{faithful} if
 $am=0$ for $a\in A$ and all $m\in M$ implies $a=0$. 
Let $f:M_1\times M_2\rightarrow N$ be a non-degenerate full $A$-bilinear mapping for some ring $A$.

Let $M$ be an $A$-module and let $\mu:A\rightarrow P$ be an inclusion of rings. Then the $P$-module $M$ is an
\textit{$P$-enrichment} of the $A$-module $M$ with respect to $\mu$ if for every $a\in A$ and $m \in M$,
$am=\mu(a)m$. Let us denote the set of all $A$ endomorphisms of the $A$-module $M$ by $End_A(M)$. Suppose the
$A$-module $M$ admits a $P$-enrichment with respect to the inclusion of rings $\mu:A\rightarrow P$. Then every
$\alpha\in P$ induces an $A$-endomorphism, $\phi_{\alpha}:M\rightarrow M$ of modules defined by
$\phi_{\alpha}(m)=\alpha m$ for $m \in M$. This in turn induces an injection $\phi_P:P\rightarrow End_A(M)$ of
rings. Thus we associate a subring of the ring $End_A(M)$ to every ring $P$ with respect to which there is an
enrichment of the $A$-module $M$.

\begin{defn}Let $f:M_1\times M_2\rightarrow N$ be a full $A$-bilinear
mapping and $\mu:A\rightarrow P$ be an inclusion of rings. The mapping $f$ admits $P$-enrichment with
respect to $\mu$ if the $A$-modules $M_1$, $M_2$ and $N$ admit $P$ enrichments with respect to $\mu$ and $f$ remains
bilinear with respect to $P$. We denote such an enrichment by $E(f,P)$.\end{defn} 

We define an ordering $\leq$
on the set of enrichments of $f$ by allowing $E(f,P_1)\leq E(f,P_2)$ if and only if $f$ as an $P_1$ bilinear
mapping admits a $P_2$ enrichment with respect to inclusion of rings $P_1\rightarrow P_2$. The largest
enrichment $E_H(f,P(f))$ is defined in the obvious way. We shall prove existence of such an enrichment for a
large class of bilinear mappings. 

The following proposition taken from~\cite{alexei86} is essential for our work.
\begin{prop}[\cite{alexei86}, Theorem 1] \label{P(f)} If $f:M_1\times M_2 \rightarrow N$ is a non-degenerate full $A$-bilinear mapping over a commutative associative ring $A$ with unit, then $f$ admits the largest enrichment.\end{prop}

\subsection{Largest ring of scalars as a logical invariant}\label{largest-ring:sec}
Indeed the ring $P(f)$ is interpretable in the bilinear map $f$ providing that $f$ satisfies certain conditions in addition to the ones in Proposition~\ref{P(f)}. 

The mapping $f$ is said to have \textit{finite width}\index{bilinear mapping!of finite width} if there is a natural number $s$ such that for every $u\in
N$ there are $x_i\in M_1$ and $y_i\in M_2$ such that
$$u=\sum_{i=1}^sf(x_i,y_i).$$
The least such number, $w(f)$\index{ $w(f)$}, is the \textit{width} of $f$.

A set $E_1=\{e_1,\ldots e_n\}$ is a \textit{left complete system} for a non-degenerate mapping $f$ if $f(E_1,y)=0$ implies $y=0$. The
cardinality of a minimal left complete system for $f$ is denoted by $c_1(f)$. A right complete system and the number $c_2(f)$ are defined correspondingly.

The \textit{type} of a bilinear mapping $f$, denoted by $\tau(f)$ \index{ $\tau(f)$}, is the triple $$(w(f),c_1(f), c_2(f)).$$ The mapping $f$ is
said to be of finite type if  $w(f)$, $c_1(f)$ and $c_2(f)$ all exist. If $f,g:M_1\times M_2 \rightarrow N$ are bilinear maps of finite type we say that the type of $g$ is less than the type of $f$ and write $\tau(g)\leq \tau(f)$ if $w(g)\leq w(f)$, $c_1(g)\leq c_1(f)$ and $c_2(g)\leq c_2(f)$.

Let $A$ be a scalar ring. Assume $M_1$, $M_2$ and $N$ are faithful $A$-modules. Let $f:M_1\times M_2\rightarrow N$ be a $A$-bilinear map. We associate two structures to $f$. The first one is
$$\mathfrak{U}(f)=\langle M_1,M_2,N, \delta\rangle.$$
where $M_1$, $M_2$ and $N$ are abelian groups and $\delta$ describes the bilinear map. The other one is  $$\mathfrak{U}_A(f)=\langle A,M_1,M_2,N,\delta,s_{M_1},s_{M_2}, s_N\rangle,$$ where $A$ is a scalar ring and $s_{M_1}$, $s_{M_2}$ and $s_N$ describe the actions of $A$ on the modules $M_1$, $M_2$
and $N$ respectively.

We state the following theorem without proof. Readers may refer to the cited reference for a proof. 

\begin{thm}[\cite{alexei86}, Theorem 2]\label{ringinter}Let $f:M_1\times M_2 \to N$ be a non-degenerate full bilinear mapping of finite type and let $P(f)$ be the largest ring of scalars of $f$. Then $\mathfrak{U}_{P(f)}(f)$
 is absolutely interpretable in $\mathfrak{U}(f)$. Moreover the same formulas interpret $\mathfrak{U}_{P(g)}(g)$ in $\mathfrak{U}(g)$ if $g$ is a full non-degenerate bilinear map with $\tau(g)\leq \tau(f)$.\end{thm}

 \subsection{Some preliminary facts on algebras}
  
 Assume $R$ is an FDZ-algebra. Since $R^+$ is a finitely generated group then $R^+$ is generated by a finite ordered set of its elements say $u_1, \ldots , u_M$ such that $U_i/U_{i+1}$ is a cyclic group generated by $u_i+U_{i+1}$ where $U_{i}$ is the subgroup generated $u_{i}, \ldots u_M$ or in notation $U_i=\langle u_{i}, \ldots u_M\rangle$. The order $e_i$ of $u_i+U_{i+1}$ is called the \emph{period} of $u_i$. If $u_i$ has infinite order then we write $e_i=\infty$. We say that $\bar{u}=(u_1,\ldots, u_M)$ is a \emph{pseudo-basis} of period  $(e_1, \ldots , e_M)$. If $e_i<\infty$ there are fixed integers $t_{ik}$ such that 
 $e_iu_i=\sum_{i+1}^M t_{ik}u_k.$ These $t_{ik}$ are called \emph{torsion structure constants} associated to $\bar{u}$. We assume an arbitrary but fixed order on the $t_{ik}$. It's an easy corollary of the structure theorem for finitely generated groups that the number $M$, period $\bar{e}$ and the structure constants $t_{ik}$ uniquely determine $R^+$ up to isomorphism. 
 
 Now consider the ring structure of an FDZ-algebra $R$ and consider a pseudo-basis $\bar{u}$ as above. Then there are fixed integer constants $t_{ijk}$ such that
 $$u_iu_j=\sum_{k=1}^M t_{ijk}u_k.$$
 The numbers $t_{ijk}$ are called the \emph{multiplicative structure constants associated to $\bar{u}$}. Again we assume a fixed order on the set of all $t_{ijk}$ obtained as above. Now it is an elementary exercise to check that the number $M$, periods $\bar{e}$, the constants $t_{ik}$ and the $t_{ijk}$ fix the ring $R$ up to isomorphism of rings. 
 
\subsubsection{Largest ring of scalars $A(R)$} Let $R$ be an $A$-algebra where $A$ is a scalar ring. Here we only consider those algebras which are faithful with respect to the action of their rings of scalars. Let $\mu: A\to A_1$ be an inclusion of rings. We say that an $A$-algebra $R$ has an $A_1$-enrichment with respect to $\mu$ if $R$ is an $A_1$-algebra and $ \alpha r =\mu(\alpha)r$, $r\in R$, $\alpha \in A$.  
 
Denote by $A(R)$ the largest, in the sense defined just above, commutative subring of $End_A(R/Ann(R))$ that satisfies the following conditions:
\begin{enumerate}
\item $R/Ann(R)$ and $R^2$ are faithful $A(R)$-modules.
\item The full non-degenerate bilinear mapping $$f_F: R/Ann(R)\times R/Ann(R)\to R^2$$ induced by the product in $R$ is $A(R)$-bilinear.
\item The canonical homomorphism $\eta:R^2 \to R/Ann(R)$ is $A(R)$-linear.
\end{enumerate}

\begin{prop}[\cite{M90c}, Proposition 8]\label{A(R)exists:prop} For any algebra $R$ the ring $A(R)$ is definable, it is unique, and does not depend on the choice of the initial ring of scalars.\end{prop}

\section{Elementary equivalence of FDZ-scalar rings}\label{Z-interpret:sec}
In this section we describe by first-order formulas some algebraic
invariants of any scalar ring $A$ with finitely
generated additive group $A^+$. In particular we provide a proof of  Theorem~\ref{elemmod:thm}.
\subsection{Interpretability of decomposition of zero into the product
of prime ideals with fixed characteristic}\label{Prime-decom:sec}

Let $A$ be a scalar ring. Suppose that we have a decomposition of zero into the product of
finitely generated prime ideals:
$$0=\p_1\cdot \p_2 \cdots \p_m,\qquad (\mathfrak{P})$$
Let $Char(\p_i)=\lambda_i$ be the characteristic of the integral
domain $A/\p_i$ and $$Char(\mathfrak{P})=(\lambda_1,\ldots ,
\lambda_m).$$ The purpose of this subsection is to obtain a formula
interpreting the decomposition of type $(\mathfrak{P})$ in $A$ with the
fixed characteristic $Char(\mathfrak{P})$, where the interpretation is uniform with respect to $Th(A)$.

A sequence of lemmas will follow. We omit some proofs as they are obvious.

\begin{lem} \label{P1} Consider the formula
$$Id(x,\bar{y})=\exists z_1,\ldots, \exists z_n (x=y_1z_1+\ldots + y_nz_n).$$
For any tuple $\bar{a}=(a_1,\ldots, a_n)\in A^n$ the formula $Id(x,\bar{a})$
defines in $A$ the ideal $id(\bar{a})$, generated by the elements
$a_1,\ldots, a_n$.\end{lem}

\begin{lem}\label{P2} The formula
$$P(\bar{y})=\forall x_1,\forall x_2(Id(x_1x_2,\bar{y})\rightarrow
(Id(x_1,\bar{y})\vee Id(x_2,\bar{y})))$$
is true for the tuple $\bar{a}$ of elements of the ring $A$ if and
only if the ideal $id(\bar{a})$ is prime.\end{lem}
\begin{lem}\label{P7} There exists a formula $Id_i(x,\bar{y}_1,\ldots , \bar{y}_i)$,
such that for any tuples $\bar{a}_1$, \ldots , $\bar{a}_i$, $Id_i(x,\bar{a}_1,\ldots , \bar{a}_i)$ defines the ideal
$\p_1\cdots \p_i$ in $A$ where $\p_k=id(\bar{a}_k)$. \end{lem}
Indeed the ideal $\p_1\cdots \p
_i$ is generated by all the products of the form $y_1\cdots y_i$
where $y_k$ is an element of the tuple $\bar{a}_k$ and the number of such products
is finite. So an application of Lemma~\ref{P1} will imply the above statement.
\begin{lem}\label{P3} The formula:
$$D(\bar{y_1},\ldots,
\bar{y_m})=\forall x( \bigwedge_{i=1}^mP(\bar{y}_i) \wedge Id_m(x,\bar{y}_1,\ldots , \bar{y}_m)\rightarrow
x=0)$$
is true for tuples $\bar{a}_1,\ldots, \bar{a}_m$ if and only if the
ideals $\p_i=Id(\bar{a}_i)$ satisfy the decomposition
$(\mathfrak{P})$.\end{lem}

\begin{lem}\label{P4} The formula
$$D_{\Lambda}(\bar{y}_1,\ldots , \bar{y}_m)=D(\bar{y}_1, \ldots
\bar{y}_m) \wedge \bigwedge_{i=1}^m \forall x Id(\lambda_ix,\bar{y}_i)\wedge\bigwedge_{i=1}^m
\exists z \neg Id(z,\bar{y}_i)$$
where $\Lambda=(\lambda_1,\ldots \lambda_m)=Char(\mathfrak{P})$ is true
for tuples $\bar{a}_1,\ldots ,\bar{a}_m$ of elements of $A$ if and only
if all the following statements hold:
\begin{itemize}
\item the ideals $\p_i=id(\bar{a}_i)$ satisfy the decomposition
$(\mathfrak{P})$,
\item if $\lambda_i>0$ then $Char(A/\p_i)=\lambda_i$, 
\item the integral domains $A/\p_i$ are all non-zero.
\end{itemize}\end{lem}

Denote by $0(\mathfrak{P})$ the number of zeros in the tuple
$(\lambda_1,\ldots , \lambda_m)$.
\begin{lem}\label{decom}Let $\mathfrak{P}=(\p_1,\ldots, \p_m)$ be a collection of finitely generated prime ideals of the scalar ring $A$, satisfying the
decomposition $(\mathfrak{P})$ and possessing the least number
$0(\mathfrak{P})$ among all such decompositions. Then for any scalar ring $B$
such that $A\equiv B$ in $L_1$, then the formula $D_{\Lambda}(\bar{y}_1,\ldots ,
\bar{y}_m)$ is true in $B$ on tuples $\bar{b}_1$, \ldots, $\bar{b}_m$,
if and only if:
\begin{enumerate}

\item $Id(x,\bar{b}_i)$ defines the prime ideal $\q_i=id(\bar{b}_i)$,
\item $0=\q_1\cdot \q_2\cdots \q_m$,
\item $Char(B/\q_i)=\lambda_i$, \quad $i=1, \ldots ,m$.
\end{enumerate}\end{lem}
\begin{proof} Items 1 and 2 follow from Lemma~\ref{P3}. If $\lambda_i>0$ then
Char$(B/q_i)=\lambda_i$ according to Lemma~\ref{P4}. Consequently,
$0(\mathfrak{P})\geq 0(\mathfrak{Q})$ where $\mathfrak{Q}=(\q_1,\ldots
, \q_m)$. If $0(\mathfrak{P})>0(\mathfrak{Q})$ then starting from
$\mathfrak{Q}$ we construct the formula $D_{\mu}$,
$\mu=char(\mathfrak{Q})$. From $A\equiv B$ and Lemma~\ref{P4} we obtain that
there exists a tuple $\mathfrak{P}'=(\p'_1,\ldots, \p'_m)$ such that
$0(\mathfrak{P}')\leq 0(\mathfrak{Q})<0(\mathfrak{P})$ which contradicts the choice of $\mathfrak{P}$. Consequently
$0(\mathfrak{P})=0(\mathfrak{Q})$ and hence
$Char(\mathfrak{P})=Char(\mathfrak{Q})$. The proposition is
proved.

\end{proof}
\begin{rem}\label{decom:rem} Any Noetherian commutative associative ring with a unit
possesses a decomposition of zero $0=\p_1\ldots \p_m$, satisfying the
assumptions of Proposition~\ref{decom}.\end{rem}

\begin{prop} \label{primary:thm} For any Noetherian associative commutative ring $A$ with a
unit, there exists an interpretable decomposition of zero into a product of prime ideals, where the interpretation is
uniform with respect to $Th(A)$.\end{prop} 
\begin{proof} The proposition is a direct corollary of Proposition~\ref{decom} and Remark \ref{decom:rem}.\end{proof}

\subsection{The case of FDZ-scalar rings}\label{elem-eq-rings:sec}
Now let $A$ be an FDZ-scalar ring. We shall denote by $r(A)$ the minimal number of generators of $A^+$ as an abelian group, say, the number of cyclic factors in the invariant decomposition of $A^+$. In case that $M$ is a finitely generated $A$-module where $A$ is as above the minimal number of generators of $M$ as an abelian group is denoted by $r(M)$, while the minimal number of generators for $M$ as an $A$-module is denoted by $r_A(M)$.

\begin{lem}\label{P9} There exists a sentence $ch_{\lambda}$ of $L_1$ such
that for any integral domain $A$ with finitely generated additive group
$A^+$:
$$char(A)=\lambda \Leftrightarrow A\models ch_{\lambda}.$$\end{lem}
\begin{proof} To prove the claim notice that if $\lambda$ is a prime then we can set
$ch_{\lambda}= \forall x (\lambda x=0)$. For $\lambda =0$ it is
enough to note that for the integral domain $A$, $char(A)=0$ if and
only if $2\neq 0$ and $1/2\notin A$. In fact if $char(A)=0$ then $2\neq
0$ and if $1/2\in A$ then $A\geq \mathbb{Z}[1/2]$ but
$\mathbb{Z}[1/2]$ is not finitely generated. Contradicting with the
assumption that $A^+$ is finitely generated. Conversely if $char(A)=p\neq 2$, then
$pA=0$. So $A$ contains the finite field $\Z/p\Z$ and so
$1/2\in A$.

\end{proof}
\begin{lem}\label{P10} Let $A$ be an FDZ-scalar ring with $r(A)=n$. Then there exists a sentence $\varphi_{n}$ of $L_1$ such that $ A
\models \varphi_{n}$ and for FDZ-scalar ring $B$,
$$B \models \varphi_n \Leftrightarrow r(B)\leq n.$$\end{lem}

\begin{proof} 
Let us first assume that $A$ is an integral domain. By  Lemma~\ref{P9} there is a sentence $ch_{\lambda}$ that defines the characteristic $char(A)=\lambda$ of $A$ in the language of rings and hence $char(B)=\lambda$ if $B\models ch_{\lambda}$. If $char(A)=\lambda\neq
0$ then $A$ is finite and $\varphi_n$ will say that $ch_\lambda$ and $A$ does not
have more than $\lambda^n$ elements. If $char(A)=0$ then
$r(A)=n$ if and only if $|A^+/2A^+|=2^n$. So in this case $\varphi_n$ will say that $ch_0$ and there are precisely $2^n$ distinct elements in $A$ modulo $2A$. 

Now assume $A$ is not necessarily an integral domain. Then $A$ is Noetherian and by Remark~\ref{decom:rem} it admits a decomposition of zero $$0=\p_1\cdot \p_2 \cdots \p_m,\qquad (\mathfrak{P})$$ with  $\Lambda=char(\mathfrak{P})$ where the prime ideals $\p_i$ are finitely generated. Set $O_i=\p_0\cdots\p_i$, where $\p_0=A$. Set also $\bar{O}_i=O_i/O_{i+1}$. Note that $r(A)$ is bounded by 
$$\sum_{i=0}^{m-1}r(\bar{O}_i).$$
So it is enough to come up with sentences $\varphi_i$ each expressing a bound for $r(\bar{O}_i)$. By lemma~\ref{P4} there are tuples of elements $\bar{a}_i$ ,$i=1, \ldots, m$ satisfying $D_{\Lambda}(\bar{a}_1, \ldots \bar{a}_m)$. Moreover if $B$ is any ring similar to $A$ with tuples of elements $\bar{b}_1, \ldots, \bar{b}_m$ which satisfy $D_\Lambda(b_1,\ldots ,b_m)$ then by Lemma~\ref{P5}, $B$ has a decomposition of zero $\mathfrak{Q}$ with same exact properties of $\mathfrak{P}$. Moreover the formula $id_i(x,\bar{a}_1, \ldots \bar{a}_i)$ from Lemma~\ref{P1} defines $O_i$ in $A$ and $id_i(x,\bar{b}_1, \ldots \bar{b}_i)$ defines similar term in $B$. The quotients  $\bar{O}_i$ are finitely generated $A$-modules over the integral domains $A/\p_i$. Assume $r(A/\p_{i+1})=n_i$ and $r_{A/\p_{i+1}}(\bar{O})=s_i$. Note that $r(\bar{O}_i)\leq n_is_i$. So it is enough to define $n_i$ and $s_i$ in the language of rings. By definability of the $\p_i$ and the $O_i$ it is easy to write a sentence in the language of rings saying that $r_{A/\p_{i+1}}(\bar{O}_i)\leq s_i$. By the first paragraph of this proof and definability of $\p_{i+1}$ there is also a sentence in the language of rings saying that $r(A/\p_{i+1})\leq n_i$. Note that the same formulas work for a ring $B$ as above. 
\end{proof}
\begin{cor} \label{10b} Assume $\mathcal{K}$ is the class of all FDZ-scalar rings. Assume $\mathcal{I}_n$ is the subclass of $\mathcal{K}$ consisting of all integral domains $A$ of characteristic zero with $r(A)\leq n$, for some natural number $n>0$. Then there exists a sentence $\Phi_n$ of the language of rings such that for any $A\in \mathcal{K}$
$$A\models \Phi_n \Leftrightarrow A\in \mathcal{I}_n.$$\end{cor}
\begin{proof} A ring $A$ being an integral domain is axiomatizable by one ring theory sentence. The formula $ch_0$ from Lemma~\ref{P9} is true in any $A\in \mathcal{I}_n$ and conversely implies that $A\in \mathcal{K}$ has characteristic 0 once $A$ satisfies it. The formula $\varphi_n$ from Lemma~\ref{P10} is satisfied by any $A\in \mathcal{I}_n$ and conversely will force $r(A)\leq n$ for any $A\in \mathcal{K}$ satisfying it. The conjunction of these sentences is the desired one.\end{proof}
\begin{cor} \label{10a} Let $A\in \mathcal{I}_n$. Then, there exists a formula $\phi_{\Z}$ of the language of rings such that 
$$A \models \phi_\Z \Leftrightarrow A\cong \Z.$$ 
\end{cor}
\begin{proof} By Corollary~\ref{10b} the formula $\Phi_1$ characterizes members of $\mathcal{I}_1$ among those of $\mathcal{K}$. But $\mathcal{I}_1$ has only one member up to isomorphism, namely $\Z$. So we may set $\phi_\Z=\Phi_1$. \end{proof}

\begin{lem}\label{julia:lem} Consider the class $\mathcal{I}_n$ introduced in Corollary~\ref{10b}. Then there exists a formula $R_n(x)$ defining
the subring $\mathbb{Z}\cdot 1_A$ in any member $A$ of $\mathcal{I}_n$. \end{lem}
\begin{proof} We need to note that the field of fractions $F$ of $A$ is an extension of field of rationals $\Q$ with dimension $n$ over $\Q$. So $F$ is a field of algebraic numbers of finite degree over $\Q$. Now by Theorem on page 956 of \cite{julia} the ring of integers $\Z$ is definable in $F$ by a formula $\Phi_{\Z(F)}(x)$. An inspection of the proof shows that the formula defines $\Z$ in any algebraic extension $K$ of $\Q$ with $[K:\Q]\leq [F:\Q]=r(A)=n$ (See the formula on line 15 of page 952 as well as the one in lines 20-21 of page 956 in~\cite{julia}). Moreover $F$ is uniformly interpretable in $A$. Though elementary, let us elaborate on this claim here a bit. Recall that $F$ is realized as $X/\sim$ where
$$X=\{(x,y):x\in A, y\in A\setminus \{0\}\},$$
and $\sim$ is the equivalence relation on $X$ defined by 
$$(x,y)\sim (z,w) \Leftrightarrow xw=yz.$$
Addition and multiplication are defined on $X/\sim$ in the obvious manner using addition and multiplication on $A$. The same formulas interpret the field of fractions $K$ of any integral domain of characteristic zero $B$ in $B$. So combining the results here we have an interpretation of $\Z$ in $A$. 

But the above interpretation of $\Z$ in $A$ also provides a formula defining $\Z$ (as a subset of $A$) in $A$ in the following way. Note that there is an interpretable monomorphism $\mu: A \to F$ defined by $\mu (a)=[(a,1)]$ where $|F|=X/\sim$ is considered as the set of equivalence classes $[(x,y)]$ described above. Now the copy of $\Z$
sitting in $F$ is included in the image of $\mu$ so the copy of $\Z$ in $A$ is a definable subset of $A$ as $\mu^{-1}(\Phi_{\Z(F)}(\mu(A))$. Since by Corollary~\ref{10a} $\Z$ is axiomatizable in $\mathcal{I}_n$ by one formula,  there exists a formula defining $\Z$ in any member of $\mathcal{I}_n$.   
\end{proof} 

\begin{lem}\label{P5} There exists a formula $R_{n,\Lambda}(x,\bar{y})$ such that for any scalar
ring $A$ with unit and $r(A)\leq n$ and for any prime ideal
$\p=id(\bar{a})$ of $A$ if $char(A/\p)=\lambda$ then the formula
$R_{n,\Lambda}(x,\bar{a})$ defines the subring
$$\mathbb{Z}\cdot 1+\p=\{z\cdot 1+x:z\in \mathbb{Z}, x\in \p\},$$
in $A$. \end{lem}

\begin{proof} Indeed the ideal $\p=id(\bar{a})$ is defined in $A$ by the formula
$Id(x,\bar{a})$. Consequently the ring $A/\p$ and the canonical
epimorphism $A \rightarrow A/\p$ are interpretable in $A$. Therefore
to obtain $R_{n,\Lambda}(x,\bar{y})$ it is sufficient to define the
subring $\mathbb{Z}\cdot 1$ in $A/\p$. In the case of $Char (A/\p)=0$ we use the formula $R_n(x)$ from Lemma~\ref{julia:lem}. As for the case of $char(A/\p)>0$ the set $\mathbb{Z}\cdot 1 +\p$ is finite
in $A/\p$ and hence definable in $A/\p$. 

\end{proof}

\begin{lem}\label{P6} Assume $A$ is a scalar ring and $r(A)\leq n$, admitting a decomposition
$$0=\p_1\ldots \p_m, \quad (\mathfrak{P})$$
into prime ideals $\p_i=id(\bar{a}_i)$ with $char(A/\p_i)=\lambda_i$, $i=1, \ldots ,
n$. Let $\Lambda=(\lambda_1, \ldots , \lambda_m)$. Then there exists a first-order formula $R_{n,\Lambda}(x, \bar{y}_1, \ldots ,
\bar{y}_m)$ of $L_1$ such that
 the formula $R_{n,\Lambda}(x,\bar{a}_1, \ldots, \bar{a}_m)$
defines in $A$ the subring
$$A_{\mathfrak{P}}=\bigcap_{i=1}^m(\mathbb{Z}\cdot 1 + \p_i),$$
\end{lem}
\begin{proof} For each $1\leq i \leq m$ consider the formula $R_{n,\lambda_i}(x,\bar{y}_i) $ introduced in Lemma~\ref{P5}. So we can set
$$R_{n,\Lambda}(x,\bar{a}_1, \ldots, \bar{a}_m)=\bigwedge_{i=1}^mR_{n,\lambda_i}(x,\bar{y}_i).$$
\end{proof} 
To a decomposition of $0$ in $A$ as above we associate the
series of ideals $$A> \p_1>\p_1\p_2>\ldots> \p_1\cdots \p_m =0$$ of the ring $A$
which will be called a $\mathfrak{P}$-series. The ring
$A_{\mathfrak{P}}$ from Lemma~\ref{P6} acts on all the quotients $\p_1\cdots \p_i/\p_1\cdots
\p_{i+1}$
 as each subring $\mathbb{Z}\cdot 1 +\p_{i+1}$ of $A$ acts on the corresponding quotient  $\p_1\cdots \p_i/\p_1\cdots
\p_{i+1}$ for each $i=1, \ldots, m$.

Recall that $L_2$ is the language of two-sorted modules. Take a module $\langle M, A\rangle$.  
If $A$ is a scalar ring admitting a decomposition of zero $\mathfrak{P}$, then $\mathfrak{P}$-series of the
ring $A$ induces a series of $A$-modules
$$M\geq \p_1M\geq \p_1\p_2M\geq \ldots \geq \p_1\cdots \p_mM=0,$$
which will also be called a \emph{$\mathfrak{P}$-series for the $A$-module $M$} or a
\emph{special series} for $M$.
The following lemma is a direct corollary of Proposition~\ref{decom}.
\begin{lem}\label{P8} There exists a formula $\phi_i(x,\bar{y}_1, \ldots ,\bar{y}_i)$ of $L_2$
such that if $\p_1\cdots \p_m=0$ is a decomposition of zero in the scalar ring
$A$ and $\p_k=id(\bar{a}_k)$, then $\phi_i(x,\bar{a}_1, \ldots ,
\bar{a}_i)$ defines the submodule $M_i=\p_1\cdots \p_iM$ in the two-sorted model $\langle M,A \rangle$.\end{lem}

The following proposition collects the main results of this section so far. 
\begin{prop}\label{big} 
Let $\la{M}$ be the class of all two-sorted modules $\langle M,A\rangle$ where $M$ is a finitely generated module over an FDZ-scalar ring $A$. Pick $\langle M, A\rangle$ in $\la{M}$ with $r(M)\leq n$. Assume there are tuples $\bar{a}_1$,$ \ldots $, $\bar{a}_m$ of elements of $A$ which satisfy $D_{\Lambda}(\bar{x}_1, \ldots , \bar{x}_m)$. Then the following hold uniformly with respect to all models $\langle N, B\rangle$ of $Th(\langle M, A\rangle)$ from $\la{M}$ which contain tuples $\bar{b}_{i}$ which satisfy $D_{\Lambda}(\bar{b}_1, \ldots , \bar{b}_m)$.
\begin{enumerate}

\item $Id(x,\bar{b}_i)$ defines in $B$ the prime ideal
$\q_i=id(\bar{b}_i)$;
\item If $\mathfrak{Q}=(\q_1, \ldots , \q_m)$ then
$char(\mathfrak{q})=\Lambda$;
\item $0=\q_1\cdots \q_m$;
\item $\phi_i(x,\bar{q}_1, \ldots , \bar{q}_i)$ defines the $i$-th
term $N_i=\q_1\cdots \q_iN$ of the special $\mathfrak{Q}$-series for $N$ in $\langle N,B \rangle$;
\item $r(B)\leq n$.
\item The formula $R_{n,\Lambda}(x,\bar{b}_1, \ldots , \bar{b}_m)$, defines the subring
$$B_{\mathfrak{Q}}=\bigcap_{i=1}^m(\mathbb{Z}\cdot 1+\q_i),$$ in $B$.

\end{enumerate}\end{prop}
\begin{proof} Items (1)-(4) follow from  Lemma~\ref{decom},
\ref{P6} and  \ref{P9}. Part (5) follows directly from Lemma~\ref{P9}. To prove (6.) we note that by (5), $r(B)\leq n$. So the statement follows from Lemma~\ref{P8}.
 \end{proof}
Finally we are ready to finish the proof of the main technical result of this section. Recall that by a $\Z$-pseudo-basis for finitely generated abelian group $M$ we simply mean a minimal generating set for $M$ as an abelian group. Assume $\bar{u}=(u_1, \ldots , u_s)$ is an ordered $\Z$-pseudo-basis for $M$ and let $M_i$ be the subgroup of $M$ generated by $u_i, \ldots, u_s$. Again, recall that the period $e_i$ of $u_i$ is the order of the cyclic group $M_i/M_{i+1}$ if $M_i/M_{i+1}$ is finite, and we set $e_i=\infty$ if the corresponding quotient is infinite. 
\begin{prop} \label{basedef} Let $\langle M, A\rangle$ and $\langle N, B\rangle$ be finitely generated modules over FDZ-scalar rings $A$ and $B$ respectively.  Let the collection of prime ideals $\mathfrak{P}=(\p_1,\ldots ,
\p_m)$, $\p_i=id(\bar{a}_i)$, satisfy the usual conditions, say as in
Proposition~\ref{big}. Assume $\bar{c}=(c_1, \ldots , c_n)$ is a $\Z$-pseudo-basis of period $\bar{f}=(f_1, \ldots f_n)$ associated to the   $\mathfrak{P}$-series of $A$, and $\bar{u}=(u_1, \ldots, u_s)$ is a pseudo-basis of period $\bar{e}=(e_1, \ldots , e_s)$ associated with the $\mathfrak{P}$-series for $M$. Then there exists a
formula $$\phi_{\mathfrak{P},n}(x_1, \ldots , x_n,y_1, \ldots , y_s, \bar{y}_1, \ldots ,
\bar{y}_m)$$ defining in the two-sorted structure $M_A^*=\langle M, A, \bar{a}_1,\ldots ,
\bar{a}_m\rangle$ the set of all $\Z$-pseudo-bases of periods $\bar{f}$ and $\bar{e}$ associated
with the $\mathfrak{P}$-series for $A$ and $M$, respectively. Moreover the formula $\phi_{\mathfrak{P},n}(\bar{x}, \bar{y}, \bar{b}_1, \ldots, \bar{b}_m)$ defines in the model $N_B^*=\langle N,B, \bar{b}_1, \ldots \bar{b}_m\rangle$ the set of all $\Z$-pseudo-bases $\bar{d}$ and $\bar{v}$ of $N$ associated to the corresponding special $\mathfrak{Q}$-series of $B$ and $N$ if $\bar{b}_1$, \ldots , $\bar{b}_m$ satisfy the formula $D_{\Lambda}$.  
\end{prop}

\begin{proof} By Proposition~\ref{big} the models
$\langle M_i, A_{\mathfrak{P}}, \bar{a}_i\rangle$ are definable in
$M_A^*$ uniformly with respect to all models $N_B^*$, with $N$ and $B$ satisfying the  hypotheses, of $Th(M_A^*)$. So it suffices to find a formula $\phi_i$ for each $i$,
defining a basis for $M_i/M_{i+1}$ of fixed period
$\bar{e}_i$.  the model
$\langle M_i, A_{\mathfrak{P}}\rangle$ is interpretable in $\langle
M,A\rangle$ with the help of any tuples of generating elements
$\bar{a}_1$, \ldots , $\bar{a}_m$ satisfying the formula
$D_{\Lambda}$. Since $M_{i+1}=\p_{i+1}M_i$, the model $\langle \ov{M}_i,A_i
\rangle$ where $\ov{M}_i=M_i/M_{i+1}$ and
$A_i=A_{\mathfrak{P}}/(\p_{i+1}\cap A_{\mathfrak{P}})$ is
obviously interpretable in $\langle M_i, A_{\mathfrak{P}}\rangle$ with the
help of $\bar{a}_{i+1}$. In the view of the fact that $A_i$ is
either $\mathbb{Z}$ or the finite field $\mathbb{Z}/p\mathbb{Z}$ and
the action of $A_i$ on
$\ov{M}_i$ is interpretable in $\langle M,A\rangle$ it is easy to write out a formula
defining all bases of $\ov{M}_i$ of given period $\bar{e_i}$ and
thus to construct the desired formula $\phi_i$. Again $\phi_i$ depends on the tuples $a_i$ as far as they satisfy $D_\Lambda$. So again by Proposition~\ref{big} the formulas $\phi_i$ define all $\Z$-pseudo-bases for the $N_i/N_{i+1}$ in a model $N_B^*$ of $Th(M_A^*)$  where the $N_i$ are defined in $N^*_B$ with the same formulas that define the $M_i$  in $M^*_A$.  
\end{proof}
Assume $A$ and $M$ satisfy the usual conditions, $\bar{c}=(c_1, \ldots , c_n)$ is a $\Z$-pseudo-basis of $A$ of period $\bar{f}=(f_1, \ldots f_n)$ and $\bar{u}=(u_1, \ldots , u_s)$ is a $\Z$-pseudo-basis of $M$ of period $\bar{e}=(e_1, \ldots , e_s)$. Then
\begin{enumerate}
\item for any $1\leq i,j\leq n$ there exist integers $s_k(c_i,c_j)$ such that $c_ic_j=\sum_{k=1}^ns_k(c_i,c_j)c_k$, 
\item for any $1\leq i \leq n$ and $1\leq j \leq s$ there exist integers $s'_k(c_i,u_j)$ such that $c_iu_j=\sum_{k=1}^ss'_k(c_i,u_j)u_k$,
\item for any $1 \leq i \leq n$ if $f_i< \infty $ then there exist integers $t_k(f_ic_i)$ such that $f_ia_i=\sum_{k=1}^n t_k(f_ic_i)c_i$,
\item for any $1\leq i \leq s$ if $e_i< \infty$ then there exist integers $t'_k(e_iu_i)$ such that $e_iu_i=\sum_{k=1}^s t'_k(e_iu_i)u_i$.
\end{enumerate}
The integers introduced above are called \emph{the structural constants} associated to the pseudo-bases $\bar{a}$ and $\bar{u}$. We assume an arbitrary but fixed ordering on the set of structure constants. It is easy to verify that $\langle M, A \rangle$ is determined up to isomorphism, as a two-sorted module, by the periods $\bar{e}$, $\bar{f}$ and the associated structure constants.

 Finally we are ready to complete the proof of Theorem~\ref{elemmod:thm}. 
 
 \noindent\emph{Proof of Theorem~\ref{elemmod:thm}.} From Proposition~\ref{basedef} we have the formula
$\varphi_{\mathfrak{P},n}$ which defines all $\mathbb{Z}$-pseudo-bases $\bar{c}$ and $\bar{u}$ of periods $\bar{f}$ and $\bar{e}$ for $A$ and $M$, respectively, in
$\langle M,A\rangle$. Again by Proposition~\ref{basedef} the same formula defines in $\langle N,B\rangle$ similar $\Z$-pseudo-bases $\bar{d}$ and $\bar{v}$ of $B$ and $N$. We need only to describe the structural constants associated with the pseudo-bases $\bar{c}$ and $\bar{u}$ for $A$ and $M$ respectively. This can be done by a formula, say $\psi_{A,M}$, of the language $L_2$ because all these constants are
integers and there are only finitely many of them. Obviously this implies that the $\Z$-pseudo-bases $(\bar{u},\bar{c})$ and $(\bar{v},\bar{d})$ are $\Z$-pseudo-bases of  
$\langle M,A\rangle$ and $\langle N, B\rangle$ respectively of the same periods and structure constants. So the theorem follows.

\qed

\section{Elementary equivalence of FDZ-algebras}\label{main:sec}

Finally here we prove the main theorem of this paper. Let us first put together a proof of Theorem~\ref{elem-iso-alg:thm}.

\noindent\emph{Proof of Theorem~\ref{elem-iso-alg:thm}.} The proof is entirely similar to that of Theorem~\ref{elemmod:thm}. In addition to the structure constants listed in items (1)-(4) above for a two-sorted module we need to describe the structure constants defining the multiplication for the ring $C$. Keeping the same notation as the proof of mentioned theorem and replacing $M$ by $C$ we need to describe the integers $t''_k(u_iu_j)$ for every $1\leq i,j\leq s$, where $u_iu_j= \sum_{k=1}^s t''_k(u_iu_j)u_k$. Again these new structure constants are also integers and could be captured in the first-order theory of $C$. The new structure constants together with the ones from items (1)-(4) above describe $\langle C, A\rangle$ up to isomorphism by a single first-order formula $\phi_{C,A}$ of $Th(\langle C,A\rangle)$. Clearly if an algebra $\langle D, B\rangle$ from $\mathcal{A}$ satisfies $\phi_{C,A}$ then $\langle C, A \rangle\cong \langle D, B\rangle.$   
\qed

We recall some notation and introduce some new ones. For an FDZ-algebra $R$ define $M(R)\define Is(R^2+Ann(R))$ and $N(R)\define Is(R^2)+Ann(R)$. Note that that $M(R)/N(R)$ is a finite abelian group. 
\begin{lem} The ideals $M(R)$ and $N(R)$ are uniformly definable in $R$ if $R$ is an FDZ-algebra.\end{lem}
\begin{proof} The ideal $Ann(R)$ is clearly uniformly definable. The ideal $R^2$ is uniformly definable among all FDZ-models of $Th(R)$, since $R^2$ has finite width.

Assume $I$ is uniformly definable in $R$. first $Is(I)/I$ is an abelian group of finite order, say $m$. Then the formula expressing $mx\in I$ uniformly defines $Is(I)$ among all FDZ-models of $Th(R)$. That is because they need to satisfy the following sentences for all $n\in \N\setminus \{0\}$.
\begin{equation}\label{TypeIso:eqn} \Psi_n: \forall x (nx \in I \to mx\in I).\end{equation}\end{proof}

\begin{lem}\label{A(R)-inter:lem} For an FDZ-algebra $R$ the maximal ring of scalars $A(R)$ and its actions on $R/Ann(R)$ and $R^2$ are absolutely interpretable in $R$.\end{lem} 
\begin{proof} The full non-degenerate bilinear map $f_{RF}:R/Ann(R)\times R/Ann(R) \to R^2$ induced by the product in $R$ is absolutely interpretable in $R$. So $P=P(f_{RF})$ an its actions on $R/Ann(R)$ and $R^2$ are interpretable in $R$ by Theorem~\ref{ringinter} since $f_{RF}$ is full, non-degenerate and of finite type. Note that 
$$A(R)=\{\alpha \in P: (\alpha x) + Ann(R)=\alpha(x+Ann(R)), \forall x\in R^2 \}.$$
Indeed $A(R)$ is clearly a definable unitary subring of $P$. This finishes the proof. \end{proof}
\begin{cor} \label{zinter:cor} If $R\equiv S$ are FDZ-algebras, then 
\begin{enumerate}
\item $A(R)\cong A(S)$ 
\item $R/Ann(R)\cong S/Ann(S)$ 
\item  $R^2 \cong S^2$
\end{enumerate}\end{cor}   
\begin{proof} To prove (1) note that since $A(R)$ and $A(S)$ are interpreted in $R$ and $S$ with the same formulas $A(R)\equiv A(S)$. Since $R$ and $S$ are FDZ-algebras $A(R)$ and $A(S)$ are FDZ-scalar rings. So by Corollary~\ref{scalarrings:cor} $A(R)\cong A(S)$. To prove (2) we note that by Lemma~\ref{A(R)-inter:lem}, the two sorted algebra $\langle R/Ann(R), A(R) \rangle$ is absolutely interpretable in $R/Ann(R)$. So indeed $\langle R/Ann(R), A(R) \rangle\equiv \langle  S/Ann(S), A(S)\rangle$. By Theorem~\ref{elem-iso-alg:thm} we have $\langle R/Ann(R), A(R) \rangle\cong \langle S/Ann(S),  A(S)$. In particular this implies $$R/Ann(R)\cong S/Ann(S)$$ as rings. (3) is similar to (2).\end{proof} 

\begin{lem}\label{mainalgR:lem} Let $R$ be a FDZ-algebra and assume $\bar{u}$ is a pseudo-basis of $R$ adapted to the series 	\begin{equation}\label{def-series:eq} R \geq  M(R) \geq N(R) \geq Is(R^2) \geq 0.\end{equation}
Then there exists a formula $\Phi(\bar{x})$ of the language rings and some fixed integers $0<l<m<n<r$, such that $R\models \Phi(\bar{u})$ and $\Phi(\bar{u})$ expresses that
\begin{enumerate}
	\item 
	\begin{enumerate}
	\item $u_1+M(R), \ldots , u_{l-1}+M(R)$ is a basis of the free abelian group $R/M(R)$,
	\item $u_{l}+Is(R^2), \ldots , u_{m-1}+Is(R^2)$ is a pseudo-basis of the finite abelian group $M(R)/N(R)$ providing the invariant factor decomposition for $M(R)/N(R)$, i.e.
		$$\frac{M(R)}{N(R)}\cong \frac{\Z}{e_l\Z}\oplus\cdots\oplus \frac{\Z}{e_{m-1}\Z},$$
		where $e_l|e_{l+1}|\cdots |e_{m-1}$, 
		\item The images of $u_m=e_lu_l$,$\ldots$, $u_{m+p-1}=e_{m-1}u_{m-1}, u_{m+p}$, $\ldots$, $u_{n-1}$ in $N(R)/Is(R^2)$, where $p=m-l$, generate a subgroup of index, say $d$, prime to $e=e_{l}\cdots e_{m-1}$ of $N(R)/Is(R^2)$. 
	\item $u_{n}, \ldots , u_l$ is a pseudo-basis of $Is(R^2)$,
\end{enumerate}
	\item  for all $1\leq i,j\leq r$ and some fixed integers $t_{ijk}$, $\ds u_iu_j=\sum_{k=n}^rt_{ijk}u_k$,
	\item for those $l\leq i \leq r$ such that $e_i < \infty$ and some fixed $t_{ik}$, $\ds e_iu_i=\sum_{k=m}^rt_{ik}u_k$ In particular for $l\leq i \leq m-1$,
	$$e_iu_i = u_{m+i}+\sum_{k=n}^rt_{ik}u_k.$$
\end{enumerate}
\end{lem}
\begin{proof} Pick a pseudo-basis of $R$ as in the statement. Note that $Ann(R)\leq M(R)$ and consider the canonical epimorphism $\theta: R/Ann(R)\to R/M(R)$. Then $\theta$ is interpretable in $R$ and so by proof of Corollary~\ref{zinter:cor} $\langle R/M(R), \Z \rangle$ is interpretable in $R$. The same holds for $Is(R^2)$, i.e. $\langle Is(R^2),  \Z \rangle$ is interpretable in $R$. The quotient $M(R)/N(R)$ is finite. So there only remains one gap $N(R)\leq Is(R^2)$ on the quotient of which the action of the ring $\Z$ is not necessarily interpretable in $R$. Set $P(R)\define N(R)/Is(R^2)$. For any integer $e\geq 2$ still the fact that $P/eP$ has a basis consisting of the images of elements $u_{m}, \ldots , u_{n-1}$ in $P/eP$ is expressible by first-order formulas. Consequently the fact that $u_m+ Is(R^2), \ldots, u_{n-1}+Is(R^2)$ generate a subgroup of index, say $d$, relatively prime to $e$ in $N(R)+Is(R^2)$ is a first-order property. Here for $e$ we pick the order of the finite group $M(R)/N(R)$, i.e. 
$$e=e_l\cdots e_{m-1}.$$
The reason for this choice will be made clear in the next lemma.

The structure constants $t_{ijk}$ are $t_{ik}$ are fixed integers and depend only on $\bar{u}$ and $R$.\end{proof}

\begin{thm} \label{mainalgS:thm} Assume $R\equiv S$ and $S\models \Phi(\bar{v})$ where $\Phi(\bar{x})$ is the formula obtained in Lemma~\ref{mainalgR:lem}. Then the following hold:
	
	 	\begin{enumerate}
	 		\item 
	 		\begin{enumerate}
	 		\item $v_1+M(S), \ldots , v_{l-1}+M(S)$ is a basis of the free abelian group $S/M(S)$, 	
	 		\item $v_l+N(S), \ldots , v_{m-1}+N(S)$ is a pseudo-basis of the finite abelian group 
	 			 				$M(S)/N(S)$ 
	 		\item  There are elements $w_{m},\ldots ,w_{n-1}$ of $S$ and integers $d_m, \ldots, d_{n-1}$ such that $w_{m}+Is(S^2),\ldots ,w_{n-1}+Is(S^2)$ is a basis of $M(S)/Is(S^2)$ and 	
	 		\begin{align*}\langle d_{m}w_{m}+Is(S^2)&,\ldots ,d_{n-1}w_{n-1}+Is(S^2)\rangle \\
	 		&= \langle v_m+Is(S^2),\ldots ,v_{n-1}+Is(S^2)\rangle\end{align*}
	 			 				 and $gcd(d,e)=1$,  where $d=d_m\cdots d_{n-1}$, and $e=e_l\cdots e_{m-1}$.	 				
	 		 \item $v_n, \ldots , v_M$ is a pseudo-basis of $Is(R^2)$. 
	 		
	 		\end{enumerate}

	      \item for all $1\leq i,j\leq r$ and some fixed integers $t_{ijk}$, $\ds v_iv_j=\sum_{k=n}^rt_{ijk}v_k$.
	 	\item for those $l\leq i \leq r$ such that $e_i < \infty$ and some fixed $t_{ik}$, $\ds e_iv_i=\sum_{k=m}^rt_{ik}v_k$.  In particular if $l\leq i \leq m-1$ then
	 	$$e_iv_i = v_{m+i} +\sum_{k=n}^r t_{ik}v_k.$$
	 	 \end{enumerate}
	\end{thm} 
	
	\begin{proof} The fact that the $v_i$ satisfy 1.(a),1.(b) and 1.(d) is corollary of the statements from Lemma~\ref{mainalgR:lem} and uniformity of the interpretations. For 1.(c) note that by 1.(c) of Lemma~\ref{mainalgR:lem} the $v_i+Is(S^2)$, $i=m, \ldots ,n-1$ will in general generate a subgroup $Q(S)$ of $P(S)=M(S)/Is(S^2)$ of finite index. By the structure theorem for finitely generated abelian groups there is a basis $\{w_i+Is(S^2): m \leq i\leq n-1\}$ of $P(S)$, and integers $d_i$,  $i=m,\ldots n-1$, such that $d_iw_i+Is(S^2)$ is a basis of $Q(S)$. So $d=d_{m}\cdots d_{n-1}$ is the index of $Q(S)$ in $P(S)$. Recall that the images of the above $w_i$'s have to form a pseudo-basis of $P(S)/eP(S)$. So one can easily check that $gcd(e,d)=1$.
		 
	 For (2) and (3) everything is clear. However the constants $t_{ijk}$ and $t_{ik}$ will not necessarily determine $S$ up to isomorphisms since $\bar{v}$ in general will generate only a subring (of finite-index as an abelian group) of $S$, clear from 1.(c).   
	 
		\end{proof}
\begin{lem}\label{regular-M=N:lem} The following are equivalent for an FDZ-algebra.
\begin{enumerate}
\item $R$ is regular. 
\item For any addition $R_0$, $R\cong R/R_0\times R_0$.
\item $M(R)=N(R)$. 
\end{enumerate}
\end{lem}  
\begin{proof} The equivalence of (1) and (2) is clear. 

Let us show that (2) implies (3). So assume $x\in M(R)$. Then for some non-zero $m\in \N$, $mx\in R^2+Ann(R)=R^2\times R_0$. Since $R=R_F \times R_0$ there exists unique $y\in R_F$ and $z\in R_0$ such that $x=y+z$. Since $mx\in R^2 \times R_0$ and $mz\in R_0$ there exist $y_1\in R^2$ and $y_2\in R_0$ such that $my=y_1+y_2$. Since $R^2\leq R_F$ and $my\in R_F$ we have $y_2\in R_F$. Therefore $y_2=0$, $my=y_1$, and $y\in Is(R^2)$. So $x\in N(R)$. 

It remains to show $(3)\Rightarrow (1)$. Consider the canonical map $\pi:R \to R/Is(R^2)$. Since $M(R)=N(R)$, there exists a direct complement $C$ for $N(R)/Is(R^2)$ in $R/Is(R^2)$. It is easy to see that $R=\pi^{-1}(C)\times R_0$ for any addition $R_0$. Since multiplication in $R/Is(R^2)$ is trivial $\pi^{-1}(C)$ is indeed a subring of $R$ and it clearly contains $R^2$.  \end{proof}

\emph{Proof of $ (1) \Rightarrow (2)$ of Theorem~\ref{mainnice:thm}.}  We follow terminology and notation of Lemma~\ref{mainalgR:lem} and Theorem~\ref{mainalgS:thm}. Indeed we will show that $S$ fits the description in Theorem~\ref{mainalgS:thm} if and only if all the conditions in the statement of Theorem~\ref{mainnice:thm} are satisfied by $S$. Let us first consider the case $M(R)\neq N(R)$.  Indeed if $S$ satisfies conditions of Theorem~\ref{mainalgS:thm} then the assignment $u_i \mapsto v_i$, $i=1, \ldots , M$ will extend to a monomorphism $\phi$ of rings since $\bar{u}$ and $\bar{v}$ are pseudo-basis of the same length, periods and structure constants. Only $im(\phi)$ which is the subring of $S$ generated by the $v_i$ may not contain all of $S$. Conversely if $\phi:R\to S$ satisfies the conditions of Theorem~\ref{mainnice:thm} it is clear that the $u_i\in R$ with description from Lemma~\ref{mainalgR:lem} will map under $\phi$ to some $v_i\in S$ as in the Theorem~\ref{mainalgS:thm}.

Now if $M(R)=N(R)$ then $M(S)=N(S)$. By Lemma~\ref{regular-M=N:lem} $X\cong X/X_0\times X_0$, where $X=R,S$ for any additions $R_0$ and $S_0$ of $R$ and $S$ respectively. Now $u_i\mapsto v_i$, $i \neq m, \ldots ,n-1$ will induce an isomorphism between $R/R_0$ and $S/S_0$ while $R_0\cong S_0$ since they are both free abelian groups of the same finite rank.  \qed

\begin{rem} \label{nutral:lem} Note that the elements $v_{m+p}, \ldots, v_n$ and the corresponding $u_i$'s, aside the rank of the subrings they generate which are just abelian groups with zero multiplication, will play no structural role in either of the rings $R$ and $S$ and split from them. Indeed all the structure constants $t_{ik}=0$ and $t_{ijk}=0$ if any of $i$,$j$ or $k$ is between $m+p$ and $n-1$.\end{rem}

\emph{Proof of Theorem~\ref{tame:thm}.} Since By assumption $Ann(R)\leq Is(R^2)$ then for any addition $R_0$ we have $R_0=0$. Now we get a series:
$$R \geq Is(R^2) \geq Ann(R)+R^2 \geq R^2 \geq 0.$$
The only problem is that even though $Is(R^2)$ is definable in $R$ in order for the corresponding formula to define $Is(S^2)$ in an FDZ-algebra $S$, $S$ has to satisfy the infinite type $\{\Psi_n:n\in \N^+\}$ from Equation~\eqref{TypeIso:eqn}. Note that $Ann(X)$ is definable in any algebra $X$ by the same sentence, while $R^2$ is definable in $R$ and the same formula defines $S^2$ is any algebra satisfying the sentence $\phi_w$ in Equation~\eqref{R^2Sen:eqn}. Now assume the order of $Is(R^2)/R^2$ is $q$. Then there exists a first-order sentence $\phi_{Is(R^2)}$ true in $R$ which will imply  $\forall x ( x \in Ann(S) \to qx \in S^2)$ in any FDZ-algebra $S$ satisfying it and therefore implying that $Ann(S)\leq Is(S^2)$ and therefore that $(Ann(S)+S^2)/S^2$ is finite. So indeed to prove the theorem we need to deal with the gaps in the following series:
$$R \geq Ann(R)+R^2 \geq R^2 \geq 0.$$ 
Recall from Corollary~\ref{zinter:cor} that $\langle R/Ann(R), \Z\rangle$ and $\langle R^2, \Z \rangle$ are interpretable in $R$, while by Theorem~\ref{elem-iso-alg:thm} each of them is axiomatized in the class of two-sorted FDZ-algebras by one sentence. It follows now that $\langle R/(Ann(R)+R^2), \Z\rangle$ is interpretable in $R$ while the former was axiomatizable by one sentence.  Therefore the isomorphism type of each quotient coming from the sequence
$$R \geq  Ann(R)+R^2 \geq R^2 \geq 0$$
can be captured by one sentence of $L$. This implies the statement.\qed 
\section{The converse of the characterization theorem}\label{converse:sec}

In this section we prove the converse of Theorem~\ref{mainalgS:thm} and therefore we shall provide a proof for $(2) \Rightarrow (1)$ direction of  Theorem~\ref{mainnice:thm}. 

 Let us fix some notation first. Let $\D$ be a non-principal ultrafilter on an index set $I$. By $R^*$ for a ring $R$ we mean the ultrapower $R^I/\D$ of $R$. The equivalence class of $x \in R^I$ in $R^*$ is denoted by $x^*$. 
\begin{lem}\label{equality:lem} Let $R$ be an FDZ-algebra and let $\D$ be a non-principal ultrafilter on $I$. Then
\begin{enumerate}
\item $(R^2)^*= (R^*)^2$,
\item $(Is(R^2))^*=Is((R^*)^2)=Is((R^2)^*),$
\item $Ann(R^*)=(Ann(R))^*$,
\item If $R_0$ is an addition of $R$ then $(R_0)^*$ is an addition of $R^*$.
\end{enumerate}\end{lem}
\begin{proof} For (1) the inclusion $\geq$ follows from the fact that $(R^I)^2$ is generated by $xy=z$, $x, y\in R^I$, where $z(i)=x(i)y(i)\in R^2$ for $\D$-almost every $i\in I$. But then the equivalence class of $xy$ is in $(R^2)^*$. The other inclusion follows from the fact that $R^2$ is of finite width. Indeed, pick $z\in R^I$ such that the equivalence class $z^*\in (R^2)^*$, and assume that width of $R^2$ is $s$. Then for $\D$-almost every $i\in I$, $z(i)=\sum_{j=1}^s x_j(i)y_j(i)$. Define $z_j\in R^I$, $j=1,\ldots s$, by $z_j(i)=x_j(i)y_j(i)$. Obviously $z^*_j\in (R^*)^2$, $j=1, \ldots ,s$, and $z^*= z^*_1+\cdots +z^*_s$. This implies the result.      

For (2) the equality of the last two terms follows from (1.). Moreover
$$x^*\in Is((R^2)^*)\Leftrightarrow m(x^*)\in (R^2)^* \Leftrightarrow (mx)^*\in (R^2)^* \Leftrightarrow x^*\in (Is(R^2))^*.$$ 
(3) is clear. (4) is implied by (2) and (3) and the definition of an addition. 
\end{proof}

\begin{lem}\label{satabelian:lem} Assume $A$ is a free abelian group of rank $n$ with basis $v_1, \ldots ,v_n$. Let $A^*$ the ultrapower of $A$ over an $\aleph_1$-incomplete ultrafilter $\D$ and let
\begin{itemize}
\item  $(b_{ij})$ be an $n\times n$ matrix with integer entries and the determinant $det((b_{ij}))=\pm 1$
\item  $\alpha_i$, $i=1, \ldots, n$, be elements of the ultrapower $\Z^*$ of the ring of integers $\Z$ over $\D$, such that $p\nmid \alpha_i$ for any prime number $p$, where $|$ denotes division in the ring $\Z^*$.    
\end{itemize} Then there is an automorphism $\psi:A^* \to A^*$ extending $v_i\mapsto \sum_{k=1}^n{\alpha_kb_{ik}}v_k$. \end{lem} 
\begin{proof} Recall that $A^*$ is an $\aleph_1$-saturated abelian group. By the structure theory of saturated abelian groups (see either of~\cite{Szmielew} or~\cite{eklof}) there is an automorphism $\eta$ of $A^*$ such that $\eta(v_k)=\alpha_k v_k$, for each $k=1, \ldots n$. Note that the automorphism $\eta$ is not necessarily a $\Z^*$-module automorphism. However since $det(b_{ik})=\pm 1$ there is $\Z^*$-module automorphism of $A^*$ extending $v_i\mapsto \sum_{k=1}^nb_{ik}v_k$. This proves the statement.\end{proof} 
 
\begin{thm}\label{converse} Assume $R$ is a FDZ-algebra with a pseudo-basis $\bar{u}$ as in Lemma~\ref{mainalgR:lem} and $S$ an FDZ-algebra as in Theorem~\ref{mainalgS:thm}. Then $$R\equiv S.$$\end{thm}
\begin{proof}
In order to prove the statement we prove that ultrapowers $R^*=R^\mathbb{N}/\D$ and $S^*=S^{\mathbb{N}}/\D$ of $R$ and $S$ over any $\omega_1$-incomplete ultrafilter $(\mathbb{N},\D)$ are isomorphic.

  By Remark~\ref{nutral:lem}, $u_k$ and $w_k$, $k=m+p-1,\ldots, n-1$ generate zero multiplication subrings of $R$ and $S$ which split from the respective rings. So just to make notation simpler we assume that
 $$n=m+p,$$
  i.e. $n-m=m-l=p$.

Recall the definition of $Q(S)$ from Theorem~\ref{mainalgS:thm}. Let $B=(b_{ik})$, $l\leq i\leq m-1$, $m\leq k\leq n-1$ be the $(m-l)\times (m-l)$ change of basis matrix between the bases $v_k+Is(S^2)$ and $d_kw_k+Is(S^2)$, i.e. $$v_i=\sum_{k=m}^{n-1}b_{ik}d_kw_k.$$  
  Now recall that $e=e_{l}\cdots e_{m-1}$, $d=d_{m}\cdots d_{n-1}$ and gcd$(d,e)=1$. Assume $\pi$ denotes the set of all prime numbers and that, $\pi_k$ is the set of all prime numbers $p$, such that $p|d_k$, $l\leq k \leq m-1$. Let us denote the $j$'th prime number in $\pi\setminus \pi_k$ by $p_{kj}$ and the product of the first $j$ primes in $\pi\setminus \pi_k$ by $q_{kj}$.
 
 For each $j\in \N$ and for all $l\leq i \leq m-1$ define 
 $$w_{ij}=\sum_{k=m}^{n-1}b_{ik}(d_k+q_{kj}e)w_k.$$
 Now let, $w_i^*\in H^*$ and $q_k^*\in \Z^*$ denote the classes of $(w_{ij})_{j\in \N}$ and $(q_{kj})_{j\in \N}$ respectively. Indeed  \begin{equation}
 \label{wi*:eqn}w^*_i=\sum_{k=m}^{n-1}b_{ik}(d_k+q_k^*e)w_k.
 \end{equation} Let us set $\alpha_k=d_k+q_k^*e$ for each relevant $k$.
 
 Next we claim that the $\alpha_k$ satisfy hypothesis (b) of Lemma~\ref{satabelian:lem}, that is, no prime $p$ divides $\alpha_k=d_k+q_k^*e$ for each $k$,  $k=m, \ldots, n-1$. To prove this we recall that $q_{kj}=p_{k1}\cdots p_{kj}$ where the $p_{k1}, \ldots ,p_{kj}$ are the first $j$ primes that do not divide $d_k$. Pick a prime $p$. If $p\in \pi_d$, i.e. $p|d_k$ and $p|(d_k+q_{kj}e)$, then $p|q_{kj}e$ which contradicts the choice of $q_{kj}$ and the fact that gcd$(d_k,e)=1$. So for such $p$, $p\nmid (d_k+q_{kj}e)$. Now pick a prime $p\in \pi\setminus \pi_k$, i.e $p\nmid d_{k}$. Then $p=p_{kt}$ for some $t\in \N$, meaning that $p$ is a factor of $q_{kj}$ for every $j\geq t$. So $p|q_{kj}e$ for every $j\geq t$. Therefore, for every such $j$ if $p|(d_k+q_{kj}e)$ then $p|d_k$, which is impossible. So for every $j\geq t$, $p\nmid d_k+q_{kj}e$. So indeed for any prime $p$, $p\nmid (d_k+q_{k}^*e)$.

 Let $R_0$ ($S_0$) be the addition of $R$ (resp. $S$) generated by $u_i$ (resp. $v_i$), $i=m, \ldots , n-1$. By Lemma~\ref{equality:lem} the $\Z^*$-submodule $R^*_0$ ($S^*_0$) of $Ann(R^*)$ ($Ann(S^*)$) generated be the $u_i$ ($v_i$), $i=m, \ldots , n-1$ is an addition of $R^*$ (resp. $S^*$) and $R^*_0=(R^*)_0=(R_0)^*$ (the same in $S$). By Lemma~\ref{satabelian:lem} there exists an isomorphism $\psi:R^*_0\to S^*_0$ extending $u_i \to w_i^*$.  
 
 By construction there exists a monomorphism $\phi:R \to S$ of groups such that: 
  \begin{equation*} 
   \phi(u_i)=\left\{
  \begin{array}{ll}
  v_i & \text{if } i\neq m,\ldots, n-1\\ \\
  \sum_{k=m}^{n-1}b_{ik}d_kw_k & \text{if }i= m,\ldots, n-1 \end{array}\right.
   \end{equation*}
  Actually $\phi$ defined above is the same $\phi$ as in Theorem~\ref{mainalgS:thm} only the $v_i$, $m\leq i \leq n-1$, are written with respect to the new basis of $S_0$ consisting of the $w_i$. For each $j\in \N$ one could twist the monomorphism $\phi:R\to S$ to get a new one denoted by $\phi_j:R\to S$ and defined by: 
 \begin{equation*} 
   \phi_j(u_i)=\left\{
  \begin{array}{ll}
  v_i & \text{if }i\neq l, \ldots , n-1\\ \\
  v_i+\sum_{k=m}^{n-1}q_{kj}\hat{e}_ib_{ik}w_k& \text{if }i= l,\ldots ,m-1\\ \\
  \sum_{k=m}^{n-1}(d_kc_{ik}+q_{kj}eb_{ik})w_k & \text{if }i= i_1+1,\ldots, i_1+n \end{array}\right.
   \end{equation*}
where $\hat{e}_i=e/e_i$. Again note that $R$ and $im(\phi_j)\leq S$ are generated by the pseudo-bases of the same lengths, periods and structure constants. Let $\phi^*:R^*\to S^*$ be the monomorphism induced $(\phi_j)_{j\in \N}$. 

Next consider the subring $R^*_f$ of $R^*$ generated by $$\{\alpha u_i, u_j: i\neq l,\ldots,m-1, \alpha\in \Z^*, j=l,\ldots, m-1\}.$$ We assume the same definitions in $S^*$ too. We claim  that $R^*=R^*_f + R^*_0$. Firstly $R^*$ is generated by all the $\alpha u_i$, $\alpha\in \Z^*$ in the obvious manner. All these generators belong to $R^*_f+ R^*_0$ with the possible exceptions when $i=l, \ldots m-1$. However by  Lemma~\ref{equality:lem} 
$$\frac{Is((R^*)^2)+ Ann(R^*))}{Is((R^*)^2+ Ann(R^*))}\cong \frac{Is(R^2)+Ann(R)}{Is(R^2+ Ann(R))}$$
is a finite abelian group and so we only need integer multiples of the $u_i$, $i=l, \ldots m-1$ in the generating set. This proves the claim.
  
Now given $x\in R^*$ there are $y\in R^*_f$ and $z\in R^*_0$ such that $x=y+z$. Now define a map $\eta:R^* \to S^*$ by 
$$\eta(x)= \phi^*(y)+\psi(z).$$ 
To show that $\eta$ is well-defined we need to check if $\phi^*$ and $\psi$ agree on $R^*_f\cap R^*_0$. We note that 
$$R^*_f\cap R^*_0=\langle e_iu_i: l\leq i \leq m-1\rangle,$$
i.e. the subgroup generated by the $e_iu_i$ as above. Now 
\begin{align*}
\phi^*(e_iu_i)&=e_i( v_i+\sum_{k=m}^{n-1}q^*_{k}\hat{e}_ib_{ik}w_k)\\
&= e_iv_i+ \sum_{k=m}^{n-1}q^*_{k}eb_{ik}w_k\\
&= \sum_{k=m}^{n-1}d_kb_{ik}w_k+ \sum_{k=m}^{n-1}q^*_{k}eb_{ik}w_k\\
&= w_{m+i}^*\\
&=\psi (u_{m+i})\\
&=\psi(e_iu_i).
\end{align*}
$\eta$ is a homomorphism since $\phi^*$ and $\psi$ are so and $\psi$ maps a subring of $Ann(R^*)$ into a subring of $Ann(S^*)$. It is injective since both $\phi^*$ and $\psi$ are injective and they agree on $R^*_f\cap R^*_0$. Finally  $$S^*=S^*_f + S^*_0= \phi^*(R^*_f) + S^*_0$$ and by construction $H^*_0=im(\psi)$. Therefore $\eta$ is surjective. We have proved that $\eta: R^* \to S^*$ is an isomorphism of rings and by the Keisler-Shelah's theorem, we have proved that
 $$R \equiv S.$$ 

\end{proof}


\begin{thebibliography}{99}



\bibitem{eklof}P. C. Eklof, and R. F. Fischer, The elementary theory of abelian groups, Ann. Math.  Logic, 4(2) (1972)  115-171.

\bibitem{hodges} W. Hodges, Model Theory, Encyclopedia of mathematics and its applications: V. 42,  Cambridge University Press, 1993.





\bibitem{alexei86} A. G. Myasnikov, Definable invariants of bilinear mappings, (Russian) Sibirsk. Mat. Zh. 31(1) (1990) 104-115, English trans. in: Siberian Math. J. 31(1) (1990) 89-99.
\bibitem{M90c} A. G. Myasnikov, The structure of models and a criterion for the decidability of complete theories of finite-dimensional algebras, (Russian) Izv. Akad. Nauk SSSR Ser. Mat. 53(2) (1989) 379-397; English translation in Math. USSR-Izv. 34(2) (1990) 389-407.
\bibitem{MR2} A. G. Myasnikov, V. N. Remeslennikov, Definability of the set of Mal'cev bases and elementary theories of finite-dimensional algebras I, (Russian) Sibirsk. Math. Zh., 23(5) (1982) 152-167. English transl., Siberian Math. J. 23 (1983) 711-724.

\bibitem{MR3} A. G. Myasnikov, V. N. Remeslennikov, Definability of the set of Mal'cev bases and elementary theories of finite-dimensional algebras II, (Russian) Sibirsk. Math. Zh. 24(2) (1983) 97-113. English transl., Siberian Math. J. 24 (1983) 231-246.


 

 
\bibitem{MS}A. G. Myasnikov and M. Sohrabi, Groups elementarily equivalent to a free 2-nilpotent group of finite rank, Alg. and Logic, 48(2) (2009), 115-139
\bibitem {MS2010} A. G. Myasnikov, M. Sohrabi, Groups elementarily equivalent to a free nilpotent group of finite rank, Ann. Pure Appl. Logic 162(11) (2011), 916-933
\bibitem{MS2016} A. G. Myasnikov and M. Sohrabi, $\omega$-stability and Morley rank of bilinear maps, rings and nilpotent groups, arXiv:1410.2280v2. 
groups, J. London Math. Society 44(2) (1991) 173-183.

\bibitem{julia} J. Robinson, The undecidability of algebraic rings and fields, Proc. Amer. Math. Soc. 10 (1959) 950-957.
\bibitem{Szmielew} W. Szmielew,  Elementary properties of abelian groups, Fund. Math. 41 (1955) 203-271.





























\end{thebibliography}
\end{document}